\documentclass[12pt]{amsart}
\usepackage{epsfig}
\usepackage{psfrag}
\usepackage{amsmath}
\usepackage{amssymb}
\usepackage{latexsym}
\usepackage{amscd}
\usepackage{amsfonts}
\usepackage{graphicx}

\usepackage{mathrsfs}
\usepackage{amsfonts}
\usepackage{amsopn}
\usepackage{amstext}
\usepackage{amscd}
\usepackage{amsthm}

\usepackage{graphics}
\usepackage{subfigure}
\usepackage{bbm} 
\usepackage{stmaryrd}
\usepackage{xcolor}
\usepackage{epstopdf}
\usepackage[pdftex,colorlinks,linkcolor=blue,menucolor=red,backref=false,bookmarks=true,citecolor=olive]{hyperref}

\usepackage{palatino}

\newtheorem{theorem}{Theorem}[section]

\newtheorem{corollary}{Corollary}[section]

\newtheorem{lemma}{Lemma}[section]

\numberwithin{equation}{section}

\newcommand{\RR}{\mathbb R}
\newcommand{\NN}{\mathbb N}

\newcommand{\comment}[1]{}

\DeclareMathOperator{\conv}{conv}
\DeclareMathOperator{\expect}{E}
\DeclareMathOperator{\prob}{P}

\DeclareMathOperator{\e}{e}

\DeclareMathOperator{\diff}{d}

\DeclareMathOperator{\closure}{cl}

\numberwithin{equation}{section}
\begin{document}
\title[Bound on the Convergence Rate in Optimal Sequence Alignment]{An Upper Bound on the Convergence Rate of a Second Functional in Optimal Sequence Alignment}

\author[R.A.~Hauser]{Raphael Hauser} 
\address{Raphael Hauser, 
Mathematical Institute, University of Oxford, Radcliffe Observatory Quarter, Woodstock Road, Oxford OX2 6GG, United Kingdom, and Pembroke College, St Aldates, Oxford, OX1 1DW, United Kingdom.}
\email{hauser@maths.ox.ac.uk}
\thanks{Raphael Hauser was supported by the Engineering and Physical Sciences Research Council [grant number EP/H02686X/1].}
\author[H.F.~Matzinger]{Heinrich Matzinger}
\address{Heinrich Matzinger,
School of Mathematics, Georgia Institute of Technology, 686 Cherry Street, Atlanta, GA 30332-0160 USA. Corresponding author.}
\email{matzi@math.gatech.edu}
\author[Ionel Popescu]{Ionel Popescu} 
\address{Ionel Popescu, School of Mathematics, Georgia Institute of Technology, 686 Cherry Street, Atlanta, GA 30332, USA, and  ``Simion Stoilow'' Institute of Mathematics  of Romanian Academy, 21 Calea Grivi\c tei, Bucharest, Romania.}
\email{ipopescu@math.gatech.edu,  ionel.popescu@imar.ro}
\thanks{Ionel Popescu was partially supported by a grant of the Romanian National Authority for Scientific Research, CNCS - UEFISCDI, project number PN-II-RU-TE-2011-3-0259 and Marie Curie Action Grant PIRG.GA.2009.249200.}

\subjclass{Primary 60F10; Secondary 92D20, 60K35}


\keywords{Sequence alignment, convex geometry, large deviations, pecolation theory}

\begin{abstract}
Consider finite sequences $X_{[1,n]}=X_1\dots X_n$ and $Y_{[1,n]}=Y_1\dots Y_n$ of 
length $n$, consisting of i.i.d.\ samples of random letters from a finite alphabet, and let $S$ and $T$ be chosen i.i.d.\ randomly from the unit ball in the space of symmetric scoring functions over this alphabet augmented by a gap symbol. We prove a probabilistic upper bound of linear order in $n^{0.75}$ for the deviation of the score relative to $T$ of optimal alignments with gaps of $X_{[1,n]}$ and $Y_{[1,n]}$ relative to $S$. It remains an open problem to prove a lower bound. Our result contributes to the understanding of the microstructure of 
optimal alignments relative to one given scoring function, extending a theory begun in \cite{geometry}.
\end{abstract}

\maketitle

\section{Introduction and Main Results}

The subject of this paper is concerned with the asyptotics of optimal sequence alignments for random sequences whose lengths tend to infinity. An important problem that occurs both in bioinformatics and 
in natural language processing is to decide on the homology of two (or more) finite sequences consisting of symbols from a fixed finite alphabet. A highly successful approach is to fix a {\em scoring function} and maximise the total score over the set of all alignments with gaps of the two sequences (for a precise definition, see the text below). Despite the combinatorially many alignments to be considered, the total score can be maximised in polynomial time by use of a dynamic programming recursion \cite{needleman-wunsch}. Using this approach, two sequences can be considered as homologous if the total score of their optimal alignment relative to a salient scoring function significantly exceeds the typical total score of an optimal alignment of two random sequences of the same length. Rigorous statistical tests on this basis require an understanding of relevant null models, thus giving the initial motivation for the theoretical study of optimal sequence alignments of random sequences and their total scores \cite{Vingron}. 

The purpose of this paper is to contribute to this theory by studying the following question: given two {\em symmetric} scoring functions $S$ and $T$, and given two i.i.d.\ random sequences of length $n$, does the rescaled total score (the score divided by $n$) relative to $T$ of an optimal alignment of the two sequences  relative to $S$ converge as $n$ tends to infinity, and if the answer to this question is `yes', can we bound the convergence rate? We will answer both questions in the affirmative.  Before we go into the technical details of our analysis, we introduce the necessary notation and background and give further details on the main contributions of this paper in relation to the exisiting literature. 

\subsection{Alignments with Gaps}
Let $n\in\NN$ and write $[1,n]:=\{1,\dots,n\}$. Consider two sequences of length $n$, $x_{[1,n]}:=(x_i)_{i\in[1,n]}$ and $y_{[1,n]}:=(y_j)_{j\in[1,n]}$ consisting of letters from a finite 
alphabet $\mathcal{A}$. Let us augment this alphabet by a symbol $G$ for a {\em gap} and write $\mathcal{A}^*=\mathcal{A}\cup\{G\}$. We define an {\em alignment} (with gaps) of $x_{[1,n]}$ and $y_{[1,n]}$ as a pair of increasing subsequences $(i_\ell)_{\ell\in[1,k]}$ and $(j_{\ell})_{\ell\in[1,k]}$ of $[1,n]$. For $\ell\in[1,k]$, each letter $x_{i_\ell}$ of the first sequence is then interpreted as aligned 
with the letter $y_{j_\ell}$ from the second sequence, while all remaining letters of either sequence are thought of as aligned with gaps. 

For example the pair of increasing subsequences $(\{1,5,6,8\},\{2,4,5,6\})$ of $[1,8]$ correspond to the alignment 
\begin{center}
\begin{tabular}{cccccccccccc}
$G$ &   $x_1$ &  $x_2$ & $x_3$ & $x_4$ & $G$ &  $x_5$  &  $x_6$ &  $x_7$ &  $x_8$ &  $G$ &  $G$ \\
$y_1$ &  $y_2$ &  $G$  &  $G$  &  $G$  & $y_3$  & $y_4$  & $y_5$ & $G$  & $y_6$  & $y_7$ & $y_8$
\end{tabular}
\end{center}
Note that the same subsequences also correspond to the alignment 
\begin{center}
\begin{tabular}{cccccccccccc}
$G$ &   $x_1$ &  $x_2$ & $G$  & $x_3$ & $x_4$ &  $x_5$  &  $x_6$ &  $x_7$ &  $x_8$ &  $G$ &  $G$ \\
$y_1$ &  $y_2$ &  $G$  &  $y_3$  &  $G$  & $G$  & $y_4$  & $y_5$ & $G$  & $y_6$  & $y_7$ & $y_8$
\end{tabular}
\end{center}
and other arrangements obtained by permuting the order of consecutive letters aligned with gaps, so that the pair $(\{1,5,6,8\},\{2,4,5,6\})$ represent in fact an equivalence class of alignments. By slight abuse of language, we will speak about an {\em alignment} when in fact referring to an entire equivalence class. In order to refer to the set of alignments of two sequences of length $n$, we introduce the following notation, 
\begin{align*}
\Lambda_{n,k}&:=\left\{\bigl((i_\ell)_{\ell\in[1,k]},(j_{\ell})_{\ell\in[1,k]}\bigr):
1\leq i_1<\dots<i_k\leq n, 1\leq j_1<\dots<j_k\leq n\right\},\;(k\in[0,n]),\\
\Lambda_n&:=\bigcup_{k=0}^n \Lambda_{n,k}.
\end{align*}

\subsection{Scoring Functions and Optimal Alignments}

A function $R:\mathcal{A}^*\times{\mathcal A}^*\rightarrow\RR$ will be called a {\em symmetric scoring function} if $R(\alpha, \beta)=R(\beta, \alpha)$ for all $\alpha, \beta\in\mathcal{A}^*$,  and $R(G,G)=0$. Given a symmetric scoring function $R$ and two finite sequences  $x_{[1,n]}$ and $y_{[1,n]}$ consisting of letters from the alphabet ${\mathcal A}$, we define the {\em total score} of $x_{[1,n]}$ and $y_{[1,n]}$ under an alignment $\nu=((i_\ell),(j_{\ell}))\in\Lambda_{n,k}$ as the sum of the scores of individually aligned letter pairs, 
\begin{equation*}
R_{\nu}(x_{[1,n]}, y_{[1,n]}):=\sum_{\ell=1}^k R(x_{i_{\ell}},y_{j_{\ell}})
+\sum_{i\in[1,n]\setminus\{i_{\ell}:\ell\in[1,k]\}} R(x_i,G) 
+\sum_{j\in[1,n]\setminus\{j_{\ell}:\ell\in[1,k]\}} R(G,y_j).
\end{equation*}
Note that since our definition of alignments with gaps disallows the situation where a gap is aligned with a gap, the value of $R(G,G)$ should be inconsequential. Our rationale for requiring $R(G,G)=0$ is to simplify some of our formulas, notably the norms defined in Section \ref{main results}. 

The {\em optimal alignment score} of $x_{[1,n]}$  and $y_{[1,n]}$ relative to $R$ is defined by 
\begin{equation*}
R^*(x_{[1,n]},y_{[1,n]}):=\max_{\nu\in\Lambda_n} R_{\nu}(x_{[1,n]}, y_{[1,n]}), 
\end{equation*}
while the set of {\em optimal alignments} of $x_{[1,n]}$  and $y_{[1,n]}$ relative to $R$ is the set of alignments 
\begin{equation*}
\nu^*_{R}(x_{[1,n]},y_{[1,n]}):=\left\{\nu\in\Lambda_n:\,R_{\nu}(x_{[1,n]}, y_{[1,n]})=R^*(x_{[1,n]}, y_{[1,n]})\right\}
\end{equation*}
on which the maximum is achieved. Note that in general, $\nu^*$ is not a singleton. 

\subsection{Random Sequences}\label{random sequences}

Let us now consider two sequences $(X_i)_{i\in\NN}:\Omega\rightarrow\mathcal{A}^{\NN}$ and $(Y_j)_{j\in\NN}:\Omega\rightarrow\mathcal{A}^{\NN}$, defined on some appropriate probability space $(\Omega,\mathscr{F},\prob)$ so as to consist of i.i.d.\ random letters $X_i$ (respectively $Y_i$) drawn from a fixed probability distribution over a finite alphabet $\mathcal{A}$. Let us again augment this alphabet by a symbol $G$ for a {\em gap} and write $\mathcal{A}^*=\mathcal{A}\cup\{G\}$. We write $X_{[1,n]}=(X_i)_{i=1}^n$ for the finite sequence consisting of the first $n$ terms of $(X_i)_{\NN}$ and use a similar notation for the second sequence.  

Let a symmetric scoring function $R$ be given on $\mathcal{A}^*\times{\mathcal A}^*$. The following is then a well defined random variable for any $n\in\NN$
\begin{align*}
L_{n,R}:\,\Omega&\rightarrow\RR,\\
\omega&\mapsto R^*(X_{[1,n]}(\omega),Y_{[1,n]}(\omega)), 
\end{align*}
and we write 
\begin{align*}
\nu^*_{n,R}:\Omega&\rightarrow\mathscr{P}\left(\Lambda_n\right),\\
\omega&\mapsto\nu^*_R\left(X_{[1,n]}(\omega),Y_{[1,n]}(\omega)\right)
\end{align*}
for the random set of optimal alignments of $X_{[1,n]}$ and $Y_{[1,n]}$ relative to $R$. 

It was shown in \cite{Sankoff1} that 
\begin{equation}\label{e:L1}
\dfrac{L_{n,R}}{n}\stackrel{n\rightarrow\infty}{\longrightarrow}\lambda_R\quad\text{almost surely}, 
\end{equation}
where $\lambda_R$ is some deterministic constant that depends only on $R$. In Lemma \ref{l:c:R} we give a proof that also establishes a quantitative convergence bound. 

\subsection{The Problem Setting of this Paper}

Let us now consider two different symmetric scoring functions $S$ and $T$ and investigate the total score relative to $T$ of an optimal alignment relative to $S$. Using the random sequences introduced above, we define the following random subsets of $\RR^2$, 
\begin{align*}
\text{SCORES}^n_{S,T}&:=\left\{\left(
\dfrac{S_\nu(X_{[1,n]},Y_{[1,n]})}{n},
\dfrac{T_\nu(X_{[1,n]},Y_{[1,n]})}{n}
\right):\,\nu\in\Lambda_n\right\}\\
\text{SET}^n_{S,T}&:=\closure\left(\conv\left(\text{SCORES}_{S,T}^{n}\right)\right),
\end{align*}
where $\closure(\cdot)$ denotes the topological closure in the canonical topology of $\RR^2$ and $\conv(\cdot)$ denotes the convex hull. 

Next, consider a symmetric scoring function $R=aS+bT$ given as a linear combination of $S$ and $T$. 
It follows from our definition of $\text{SET}^n_{S,T}$ that 
\begin{equation}\label{lnR}
\frac{L_{n,R}}{n}=\max_{(x,y)\in \text{SET}^n_{S,T}}f_{(a,b)}(x,y),
\end{equation}
where $f_{(a,b)}:\,(x,y)\mapsto ax+by$ is the linear form on $\RR^2$ defined by the weights $a,b$. Combining Equations \eqref{e:L1} and \eqref{lnR}, it follows that 
\begin{equation}\label{for later}
\max_{(x,y)\in \text{SET}^n_{S,T}}f_{(a,b)}(x,y)\stackrel{n\rightarrow\infty}{\longrightarrow}\lambda_{aS+bT},
\quad\text{a.s.}.
\end{equation}

We observe that, if a sequence of random compact convex sets $A_1,A_2,\dots \subset\RR^2$ has the property that for any linear functional $f\in(\RR^2)^*$,  
\begin{equation*}
\max_{(x,y)\in A_n}f(x,y)\stackrel{n\rightarrow\infty}{\longrightarrow}\xi_f,\quad\text{a.s.},
\end{equation*}
where $\xi_f\in\RR$ is a deterministic constant that depends only on $f$, then the sequence $(A_{n})_{n\in\NN}$ converges in Hausdorff distance to a convex compact set $A$. We will prove this claim in Lemma~\ref{l:2}.  For compact sets $A,B\subset\RR^{2}$, the Hausdorff distance is defined as 
\begin{equation}\label{e:Hd}
d_{H}(A,B)=\max\{ \sup_{x\in A}\inf_{y\in B}d(x,y),\sup_{y\in B}\inf_{x\in A}d(x,y)\},
\end{equation}
where $d(x,y)=\|x-y\|_2$ denotes the Euclidean distance.   

Equation \eqref{for later} and the fact that any $f\in(\RR^2)^*$ is of the form $f_{(a,b)}$ for some $(a,b)\in\RR^2$ show that the above made observation is applicable to the sequence of sets $(\text{SET}^n_{S,T})_{n\in\NN}$. There exists therefore a deterministic convex compact set 
$\text{SET}_{S,T}$ for which 
\begin{equation}\label{e:3}
d_{H}(\text{SET}_{S,T}^{n},\text{SET}_{S,T})\stackrel{n\rightarrow\infty}{\longrightarrow} 0,\quad\text{a.s.}
\end{equation}
One of our goals is to refine this analysis and quantify an upper-bound on the rate of convergence. An upper bound on the convergence was givne in \cite{geometry} for scoring functions that are not necessarily symmetric. In this paper we give a much simpler proof that is made possible by exploiting the symmetry of scoring functions. Since most scoring functions used in applications are symmetric, the simplification is of interest. 

Another goal is to study how much the total score relative to $T$ varies when two random strings are aligned optimally relative to $S$. Note that we have 
\begin{equation*}
L_{n,S}=\max_{(x,y)\in \text{SET}^n_{S,T}}x.
\end{equation*}
In general, we should not expect that $\nu^{*}_{n,S}$ to be a singleton. In other words, there may exist multiple optimal alignments of $X_{[1,n]}$ and $Y_{[1,n]}$ relative to $S$. Therefore, we need to consider the following quantities, 
\begin{align}
\max_{\pi\in\nu^*_{n,S}}\dfrac{T_\pi(X_{[1,n]},Y_{[1,n]})}{n}
&=\max\left\{y:\,(x,y)\in \text{SET}^n_{S,T},\,x=\frac{L_{n,S}}{n}\right\},\label{the max}\\
\min_{\pi\in\nu^*_{n,S}}\dfrac{T_\pi(X_{[1,n]},Y_{[1,n]})}{n}
&=\min\left\{y:\,(x,y)\in \text{SET}^n_{S,T},\,x=\frac{L_{n,S}}{n}\right\},\label{the min}\\
\end{align}
Lemma~\ref{l:6} will establish that if $\max_{(x,y)\in\text{SET}_{S,T}}x$ has a unique maximiser $(x_0,y_0)$, then the upper and lower bounds \eqref{the max}, \eqref{the min} both converge to $y_0$ almost surely. 

\subsection{Statement of the Main Results}\label{main results}

To state the main results of this paper, we introduce the following norms on the set of symmetric scoring functions $R:\mathcal{A}^*\times\mathcal{A}^*\rightarrow\RR$, 
\begin{align}
|R|&:=\max_{a,b,c\in\mathcal{A}^*}|R(a,b)-R(a,c)|,\quad\text{(the \em{change norm})},\label{norm}\\
|R|_2&:=\sqrt{\sum_{a,b\in\mathcal{A}^*}R^2(a,b)},\quad\text{(the Frobenius norm)}. \label{e:n2}
\end{align}
The change norm plays the following important role: given two finite sequences and a fixed alignment with gaps, changing a single letter of one of the two sequences into an arbitrary other letter from the alphabet $\mathcal{A}$ changes the total score of the alignment by at most $|R|$. 


\begin{theorem}\label{t:1} Let $S$ and $T$ be two symmetric scoring functions on $\mathcal{A}^*\times\mathcal{A}^*$ such that the optimisation problem 
$\max_{(x,y)\in\text{SET}^n_{S,T}}x$ has a unique maximiser $(x_0,y_0)$ and the boundary of $\text{SET}_{S,T}$ has curvature at least $k>0$ at this point, then the following bound applies 
for large enough $n$, where $\e$ is the Euler constant, 
\begin{equation*}
\prob\left[\left|\dfrac{T_\pi(X_{[1,n]},Y_{[1,n]})}{n}-y_0\right|\leq 
\dfrac{5|T|+2\sqrt{30|S|}}{k}\left(\frac{\ln(n\e)}{n} \right)^{1/4},\quad
\forall \pi \in \nu^*_{n,S}\right]\geq 
1-3n^{-\ln n}. 
\end{equation*}
\end{theorem}

In particular if both $S$ and $T$ have change norm less than $1$, the statement of Theorem \ref{t:1} 
simplifies to 
\begin{equation*}
P\left[\left|\dfrac{T_\pi(X_{[1,n]},Y_{[1,n]})}{n}-y_0\right|\leq 
\dfrac{11}{k}\left(\frac{\ln(n\e)}{n}\right)^{1/4},\quad\forall \pi \in \nu^*_{n,S}\right]\geq 
1-3n^{-\ln n},\quad\forall n\gg 1.
\end{equation*}
The curvature condition at the point $(x_0,y_0)$ means that one can parametrize the boundary 
$\partial \text{SET}_{S,T}$ of the set  $\text{SET}_{S,T}$ by a curve $c(t)$ for $t$ in a 
neighbourhood of $0$, with $c(0)=(x_0,y_0)$ and $\|\dot{c}\|_2=1$ for all $t$, where $\dot{c}$ 
denotes the derivative with respect to $t$, the curvature 
\begin{equation*}
\kappa(\partial\text{SET}_{S,T},(x_0,y_0)):=\|\ddot{c}(0)\|_2
\end{equation*}
then being defined as the standard curvature of this curve at $t=0$. By convention, we define the curvature at vertices of $\partial\text{SET}_{S,T}$ (points on the boundary where $\text{SET}_{S,T}$ has a normal cone with nonempty interior) to be $+\infty$. We postpone the proof of Theorem \ref{t:1} until Section~\ref{s:thms}.   

While Theorem \ref{t:1} establishes that if the boundary of $\text{SET}_{S,T}$ has positive curvature at $(x_0,y_0)$,  then the $T$-score on an $S$-optimal alignment has a fluctuation of order at most $O([\ln(n)/n]^{0.25})$, the conditions of this result are difficult to verify in practice. However, as the following result shows, they apply generically:

\begin{theorem}\label{t:2} 
Let $S$ and $T$ be chosen i.i.d.\ uniformly at random from the Frobenius-unit sphere in the space of symmetric scoring functions. Then the following hold true, 
\begin{enumerate}
\item $\max_{(x,y)\in\text{SET}^n_{S,T}}x$ has a unique maximiser $(x_0,y_0)$ almost surely, 
\item for any real number $k>0$,
\begin{equation*}
\prob\left[\kappa(\partial\text{SET}_{S,T},(x_0,y_0))<k\right]\leq\dfrac{4k}{\pi},
\end{equation*}
where $\kappa(\partial\text{SET}_{S,T},(x_0,y_0))$ is the curvature at $(x_0,y_0)$ of the boundary of 
$\text{SET}_{S,T}$.
\end{enumerate}
\end{theorem}

Combining Theorems ~\ref{t:1} and ~\ref{t:2}, we arrive at the following conclusion:  

\begin{corollary} If the symmetric scoring functions $S$ and $T$ are chosen as in Theorem~\ref{t:2}, then almost surely there exists $k>0$ such that  
\begin{equation*}
P\left[\left|\dfrac{T_\pi(X_{[1,n]},Y_{[1,n]})}{n}-y_0\right|\leq \frac{11}{\max(k,1)}\left(\frac{\ln(n\e)}{n}\right)^{1/4},\quad
\forall \pi \in\nu^*_{n,S}
\right]\geq 1-3n^{-\ln n},\quad\forall n\gg 1.
\end{equation*}
\end{corollary}

\section{Preliminary Results and their Proofs}\label{s:res}

In this section we derive the main estimates on which the proofs of our main theorems rely.  
We begin by giving the classical Azuma-Hoeffding -- McDiarmid Inequality.  

\begin{theorem}\label{t:AH}
Let $W_1,\dots,W_n$ i.i.d.\ random variables that take values in some set $D$, let $a>0$ be a constant and $f:D^n\rightarrow\RR$ a $n$-variate real function with the property that for any $i\in[1,n]$, 
$w\in D^n$ and $z\in D$, 
\begin{equation*}
\left|f(w_1,w_2,\dots,w_n)-f(w_1,w_2,\dots,w_{i-1},z,w_{i+1},\ldots,w_n)\right|\leq a.
\end{equation*}
Then, for any $\epsilon>0$, the following inequalities hold true, 
\begin{align*}
\prob\left[\left|f(W_1,W_2,\dots,W_n)-\expect[f(W_1,\dots,W_n)]\right|\geq\epsilon n\right]&\leq 2\exp(-\epsilon^2n/(2a^2)),\\
\prob\left[f(W_1,W_2,\dots,W_n)-\expect[f(W_1,\dots,W_n)]\geq\epsilon n\right]&\leq \exp(-\epsilon^2n/(2a^2)).
\end{align*}
\end{theorem}

For a proof, see e.g.\ \cite{azuma}.

\begin{lemma}\label{l:c:R}
For any symmetric scoring function $R:\mathcal{A}^*\times\mathcal{A}^*\rightarrow\mathbb{R}$ there exists a deterministic constant $\lambda_R$ such that 
\begin{equation*}
\dfrac{L_{n,R}}{n}\stackrel{n\rightarrow\infty}{\longrightarrow}\lambda_R,\quad
\text{a.s.}
\end{equation*}
\end{lemma}

\begin{proof}
It is trivial to see that the function $n\mapsto E[L_{n,R}]$ is superadditive. Therefore and since the scoring function is bounded, we have  
\begin{equation}\label{lambda}
E[L_{n,R}]/n\stackrel{n\rightarrow\infty}{\longrightarrow}\lambda_{R}:=\sup_{n\ge1}E[L_{n,R}]/n, 
\end{equation} 
where $\sup_{n\ge1}E[L_{n,R}]/n$ is well defined. For any $\epsilon>0$, let $D_{n,R}(\epsilon)$ denote the event
\begin{equation*}
D_{n,R}(\epsilon)=\left\{\left|L_{n,R}-E[L_{n,R}]\right|\geq \epsilon\ln(n)\sqrt{n}\right\}.
\end{equation*}
Applying Theorem \ref{t:AH} with $a=|R|$, we obtain  
\begin{equation}\label{e2:1}
\prob\left[D_{n,R}(\epsilon)\leq 2\exp(-\epsilon^{2}(\ln n)^2/2|R|^2\right]=2n^{-\frac{\epsilon^{2}\ln n}{2|R|^2}}.
\end{equation}
By virtue of Borel-Cantelli, the finite summability of \eqref{e2:1} implies that almost surely at most a finite number of the events $D_{n,R}(\epsilon)$ will hold. Combined with \eqref{lambda}, and using the fact that $\epsilon>0$ was arbitrary, this implies the claim.  
\end{proof}

The next result gives the rate of convergence for of $E[L_{n}(R)]/n$ toward $\lambda_{R}$.  A bound for non-symmetric scoring functions was given in \cite{geometry}. Here we exploit the symmetry of $R$ to give a tighter bound that we will use to prove our main theorems. 

\begin{lemma} \label{l:3}
For any symmetric scoring function $R:\mathcal{A}^*\times\mathcal{A}^*\rightarrow\RR$, the following convergence bound applies, 
\begin{equation*}
\left|\lambda_R-\expect\left[\dfrac{L_{n,R}}{n}\right]\right|\leq 
3|R|\sqrt{\dfrac{\ln(n\e)}{n}}.
\end{equation*}
\end{lemma}

\begin{proof}
To simplify the notation, let us write $\lambda_{n,R}=\expect[L_{n,R}]/n$.  Let $m=kn$ for some $k\in\NN$, and let $\mathcal{P}^{m,n}$ be the set of pairs of partitions of the integer interval $[1,m]$ into $2k$ pieces for which the sum of the lengths of the $i$-th pieces is always $n$. In other words, 
\begin{equation*}
{\mathtt{p}}=(i_0,i_1,\ldots,i_{2k}, j_0,j_1,\ldots,j_{2k})
\end{equation*}
is in $\mathcal{P}^{m,n}$ if
\begin{align*}
0&=i_0<i_1<i\dots<i_{2k}=m,\\
0&=j_0<j_1<\dots<j_{2k}=m,\quad\text{and}\\
i_\ell&-i_{\ell-1}+j_\ell-j_{\ell-1}=n,\quad\forall\,\ell\in[1,2k].
\end{align*}
For a partition $\tt{p}\in\mathcal{P}^{n,m}$, let $L^{\mathtt{p}}_{m,R}$ denote
the optimal alignment score of $X_{[1,n]}$ and $Y_{[1,n]}$ relative to $R$ under the extra constraint that the $l$-th pieces of the two partitons are aligned with each other,
hence imposing that $X_{i_{l-1}+1} \ldots X_{i_l}$ be aligned with $Y_{j_{l-1}+1}\ldots Y_{j_l}$ for $l=1,\ldots, 2k$.  In other words, we have
\begin{equation}\label{lalala}
L^{\tt{p}}_{m,R}=\sum_{l=2}^{2k} 
R(X_{i_{l-1}+1} \ldots X_{i_l},Y_{j_{l-1}+1}\ldots Y_{j_l}).
\end{equation}

We can apply Azuma-Hoeffding to our constrained optimal alignment score
$L^{\tt{p}}_{m,R}$ to justify that   for any constant $\epsilon>0$,
\begin{equation}\label{mimi}
P(L^{\tt{p}}_{m,R}-E[L^{\tt{p}}_{m,R}]\geq \epsilon m)\leq 
\exp\left(-\frac{\epsilon^2\cdot m}{2|R|^2}\right).
\end{equation}

The optimal alignment score $L_{m,R}$ is not always equal to one of the 
the constrained alignment scores $L^{\tt{p}}_{m,R}$, however we can argue that it is not far from this.  
In fact, it is not hard to see that for some partition $\tt{p}$ 
\begin{equation}\label{correction term}
|L_{m,R}- L^{\tt{p}}_{m,R}|\le 4k|R|.
\end{equation}
Therefore, if the alignment score $L_{m,R}$ is to exceed a given benchmark, at least one of the 
constrained scores $L^{\tt{p}}_{m,R}$ must exceed this benchmark shifted by the correction term 
\eqref{correction term}. This implies
\begin{equation}\label{samutu}
\prob\left[L_{m,R}\geq n\lambda_{n,R}\, k+\epsilon m\right]\leq
\sum_{\tt{p}\in\mathcal{P}^{m,n}}
\prob\left[L^{\tt{p}}_{m,R}\geq n\lambda_{n,R} m+\epsilon m-4k|R|\right].
\end{equation}

We claim that by symmetry of $R$,  we have 
\begin{equation}\label{lorraine}
\expect[L^{\tt{P}}_{m,R}]\leq n\lambda_{n,R}\, k,\quad\forall\,\tt{p}\in \mathcal{P}^{n,m}. 
\end{equation}
Our claim holds for two reasons: Firstly, $i_{l}-i_{l-1}+j_{l}-j_{l-1}=n$ implies $i_{l}< i_{l-1}+n$ and $j_{l}< j_{l-1}+n$ and 
\begin{align*}
R(X_{i_{l-1}+1} \ldots X_{i_l},Y_{j_{l-1}+1}\ldots Y_{j_l})&+
R(X_{i_{l}+1} \ldots X_{i_{l-1}+n},Y_{j_{l}+1}\ldots Y_{j_{l-1}+n})\\
&\leq R(X_{i_{l-1}+1} \ldots X_{i_{l-1}+n},Y_{j_{l-1}+1}\ldots Y_{j_{l-1+n}}).
\end{align*}
Taking expectations on both sides, we find 
\begin{align}
\expect[R(X_{i_{l-1}+1} \ldots X_{i_l},Y_{j_{l-1}+1}\ldots Y_{j_l})]&+
\expect[R(X_{i_{l}+1} \ldots X_{i_{l-1}+n},Y_{j_{l}+1}\ldots Y_{j_{l-1}+n})]\nonumber \\
&\leq\expect[R(X_{i_{l-1}+1} \ldots X_{i_{l-1}+n},Y_{j_{l-1}+1}\ldots Y_{j_{l-1+n}}].\label{moimoi}
\end{align}
Secondly, the crucial assumption that $R$ be symmetric implies  
that the two terms on the left-hand side of \eqref{moimoi} are equal, thus yielding 
\begin{align}
2\expect[R(X_{i_{l-1}+1} \ldots X_{i_l},Y_{j_{l-1}+1}\ldots Y_{j_l})]&\leq
\expect[R(X_{i_{l-1}+1} \ldots X_{i_{l-1}+n},Y_{j_{l-1}+1}\ldots Y_{j_{l-1+n}})]\nonumber\\
&=\expect[R(X_1\ldots  X_n,Y_1\ldots Y_n)]\nonumber\\
&=n\lambda_{n,R}.\label{moimoi2}
\end{align}
Taking the expectation on both sides of \eqref{lalala} and applying \eqref{moimoi2}  
to each term on the right-hand side yields the claimed inequality, \eqref{lorraine}. 

Substitution of \eqref{lorraine} into \eqref{samutu} now yields 
\begin{equation}\label{samutu2}
\prob\left[L_{m,R}\geq n\lambda_{n,R} k+\epsilon m\right]\leq
\sum_{\tt{p}\in\mathcal{P}^{m,n}}
\prob\left[L^{\tt{p}}_{m,R}\geq \expect[L{\tt{p}}_{m,R}]+\epsilon m - 4k\, |R|\right].
\end{equation}
Using \eqref{mimi} and the fact that $\mathcal{P}^{n,m}$ has fewer than $\binom{m}{k}^2$ elements
yields that for large $n$ and $k$, 
\begin{equation}\label{samutu3}
\prob\left[L_{m,R}\geq n\lambda_{n,R} k+\epsilon m\right]\leq
\binom{m}{k}^2\exp\left(-\frac{\left(\epsilon- 4|R|/n\right)^2\cdot m}{2|R|^2}\right).
\end{equation}

Let $Z$ be a binomial variable with parameters $m$ and $p=1/n$, so that we have 
\begin{equation*}
\prob[Z=k]=\binom{m}{k} \left(\frac{1}{n}\right)^k\cdot 
\left(\frac{n-1}{n}\right)^{m-k}\leq 1, 
\end{equation*}
and hence, 
\begin{equation}\label{expoln}
\binom{m}{k}\leq n^k\cdot \left(\frac{1}{1-\frac{1}{n}}\right)^{k(n-1)}
\leq (\e\cdot n)^k,\quad(n\gg 1).
\end{equation} 
Substituting \eqref{expoln} into \eqref{samutu3}, we find that for large $n$, 
\begin{equation}\label{samutu4}
\prob\left[\frac{L_{m,R}}{m}\geq\lambda_{n,R}+\epsilon \right]\leq
\exp\left(k \left[2\ln(\e\cdot n)-\dfrac{\left(\epsilon- 4|R|/n\right)^2\cdot n}{4|R|^2}\right]\right).  
\end{equation}

The key now is to let $k$ tend to infinity. In doing so, we know on the one hand that that $L_{m,R}/m\rightarrow\lambda_{R}$, and on the other that the the right-hand side of \eqref{samutu4} converges either to $0$ or $+\infty$.  It does converge to $0$ only if 
\[
2\ln(\e\cdot n)-\dfrac{\left(\epsilon- 4|R|/n\right)^2\cdot n}{4|R|^2}<0
\]
which is certainly satisfied if $n$ is chosen large enough ($n>10$ suffices) and 
\[
\epsilon=3|R|\sqrt{\frac{\ln(n\e)}{n}}.
\]
Therefore, we find  
\[
\prob\left[\lambda_{R}\geq\lambda_{n,R}+3|R|\sqrt{\frac{\ln(ne)}{n}} \right]=0,
\]   
and since $\lambda_{R}$ is a constant, and similarly $\lambda_{n,R}$, we actually deduce that 
\[
\lambda_{R}\leq\lambda_{n,R}+3|R|\sqrt{\frac{\ln(ne)}{n}}.  
\]
On the other hand, we also know from \eqref{lambda} that $\lambda_{n}/n\leq\lambda_{R}$, thus concluding the proof. 
\end{proof}

\begin{lemma}\label{l:66} Let $R:\mathcal{A}^*\times\mathcal{A}^*\rightarrow\RR$ be a 
symmetric scoring function and let $A^n(R)$ denote the event 
\begin{equation*}
A^{n}(R)=\left\{\left|\lambda_R-\frac{L_{n,R}}{n}\right|\leq 5|R|\sqrt{\frac{\ln(n\e)}{n}}\right\}.
\end{equation*}
Then for large $n$, 
\[
\prob\left[A^n(R)\right]\geq 1-n^{-\ln n}.
\]
\end{lemma}

\begin{proof}
This follows by combining \eqref{e2:1} with $\epsilon=2|R|$, Lemma~\ref{l:3}, Theorem~\ref{t:AH} and Lemma~\ref{l:c:R}.   
\end{proof}

The next result is about the convergence of convex compact sets.    

\begin{lemma}
\label{l:2}
Let  $(A_n)_{n\in\NN}$ be a sequence of random compact convex sets in $\mathbb{R}^2$ such that for any linear form $f\in(\RR^2)^*$ there exists a deterministic constant $\xi_f\in\RR$ for which 
\begin{equation*}
\max_{(x,y)\in A_n} f(x,y)\stackrel{n\rightarrow\infty}{\longrightarrow}\xi_f,\quad\text{a.s.}
\end{equation*}
Then there exists a deterministic compact convex set $A\subset\RR^2$ for which 
\begin{equation*}
d_H(A_n, A)\stackrel{n\rightarrow\infty}{\longrightarrow} 0,\quad\text{a.s.},
\end{equation*}
where $d_H$ is the Hausdorff distance.
\end{lemma}

\begin{proof}
Let $F$ be a dense countable subset of the unit sphere in $(\RR^2)^*$. Then 
\begin{equation*}
A:=\{(x,y):\,f(x,y)\leq\xi_f,\,\forall\,f\in(\RR^2)^*\}=\{(x,y):\,f(x,y)\leq\xi_f,\,\forall\,f\in F\}.
\end{equation*}
Furthermore, $A$ is compact and convex, the condition of the lemma implies that 
\begin{equation}\label{halloOxford}
\prob\left[\max_{(x,y)\in A_n} f(x,y)\stackrel{n\rightarrow\infty}{\longrightarrow}\xi_f,\,\forall\,f\in F\right]=1, 
\end{equation}
and we have 
\begin{equation}\label{cru}
\max_{(x,y)\in A}f(x,y)=\xi_f,\quad\forall\,f\in F.
\end{equation}
Suppose it is not the case that $d_H(A_n,A)\rightarrow 0$ almost surely. Then there exists $\delta>0$ and a set ${\mathscr E}\subset\Omega$ such that $\prob[\mathscr{E}]>0$ and $\forall\,\omega\in\mathscr{E}$  
there exists a sequence of points $(\alpha_{n}(\omega))_{n\in\NN}$ such that $\alpha_n(\omega)\in A_{n}(\omega)$ and  
\begin{equation*}
d(\alpha_{n}(\omega), A):=\min_{\beta\in A}d(\alpha_n(\omega), \beta)\geq\delta. 
\end{equation*}
Since all sets $A_{n}(\omega)$ are contained in some large closed box, there exists a convergent subsequence $(\alpha_{n_{k}}(\omega))_{k\in\NN}\rightarrow\alpha(\omega)$. 
The continuity of the function 
$\alpha\mapsto d(\alpha, A)$ implies that we have $d(\alpha(\omega), A)\geq\delta>0$, and in 
particular that $\alpha(\omega)\notin A$. By virtue of the Hahn-Banach separation theorem, there exists $g_\omega\in(\RR^2)^*$ such that $A\subset\{(x,y): g_{\omega}(x,y)\leq\max_{(s,t)\in A}g_{\omega}(s,t)\}$ and $g_{\omega}(\alpha)>\max_{(s,t)\in A}g_{\omega}(s,t)+\epsilon$ for some $\epsilon>0$. 
Let $(f_\ell)_{\ell\in\NN}\subset F$ be a sequence such that $f_{\ell}\rightarrow g_{\omega}$ in the weak topology. By \eqref{cru}, we have $A\subset\{(x,y): f_{\ell}(x,y)\leq\xi_{f_{\ell}}\}$, and for $\ell$ large enough it is the case that $f_{\ell}(\alpha)>\xi_{f_{\ell}}+2\epsilon/3$. If it were now the case that 
\begin{equation}\label{hypo}
\max_{(x,y)\in A_n(\omega)}f_{\ell}(x,y)\rightarrow\xi_{f_{\ell}}, 
\end{equation}
then for large enough $n$,  
\begin{equation*}
f_{\ell}(\alpha(\omega))>\xi_{f_{\ell}}+2\epsilon/3>\max_{(x,y)\in A_{n_{k}}(\omega)}f_{\ell}(x,y)+\epsilon/3\geq f_{\ell}(\alpha_{n_k}(\omega))+\epsilon/3.
\end{equation*}
But this is a contradiction, since by continuity of $f_{\ell}$, we have $f_{\ell}(\alpha_{n_{k}}(\omega))\rightarrow f(\alpha(\omega))$. We conclude that for each $\omega\in\mathscr{E}$ there exists $f_{\ell}\in F$ for which \eqref{hypo} does not apply, and since $\prob[\mathscr{E}]>0$, this contradicts \eqref{halloOxford}. 
\end{proof}

\begin{lemma}\label{l:6}
Let $S,T$ be two symmetric scoring functions on $\mathcal{A}^*\times\mathcal{A}^*$. 
If the optimization problem $\max_{(x,y)\in\text{SET}_{S,T}}x$ has a unique maximizer $(x_0,y_0)$, then 
\begin{align}
\max_{\pi\in\nu^*_{n,S}}\dfrac{T_\pi(X_{[1,n]},Y_{[1,n]})}{n}
&\stackrel{n\rightarrow\infty}{\longrightarrow}y_0,\quad\text{a.s.},\label{e:5,1}\\
\min_{\pi\in\nu^*_{n,S}}\dfrac{T_\pi(X_{[1,n]},Y_{[1,n]})}{n}
&\stackrel{n\rightarrow\infty}{\longrightarrow}y_0,\quad\text{a.s.}\label{e:5,2}
\end{align}
\end{lemma}

\begin{proof}
By virtue of \eqref{for later} and Lemma \ref{l:2}, 
$d_H(\text{SET}_{S,T}^n,\text{SET}_{S,T})\rightarrow 0$ almost surely. 
Keeping in mind \eqref{the max} and \eqref{the min}, taking any convergent subsequence 
$((x_{n_{\ell}},y_{n_{\ell}}))_{\ell\in\NN}$ of a sequence $((x_n,y_n))_{n\in\NN}$ of maximizers 
\begin{equation}\label{maximisers}
(x_n,y_n)\in\arg\max\left\{y:\,(x,y)\in\text{SET}_{S,T}^n,\,x=\dfrac{L_{n,S}}{n}\right\}, 
\end{equation}
and writing $(x^*,y^*)=\lim_{\ell\rightarrow\infty}(x_{n_{\ell}},y_{n_{\ell}})$, we have 
$x^*=x_0$ almost surely (by virtue of \eqref{for later}), and $(x^*,y^*)\in\text{SET}_{S,T}$ almost surely. By the assumptions of the lemma, we thus have $(x^*,y^*)=(x_0,y_0)$. Furthermore, a convergent subsequence of $((x_n,y_n))_{n\in\NN}$ always exists, since all sets $\text{SET}_{S,T}^n$ are contained in a compact box, and the argument above shows that $(x_0,y_0)$ is the only accumulation point. Therefore, 
$(x_n,y_n)\rightarrow(x_0,y_0)$ almost surely, and since the choice of $(x_n,y_n)$ among the maximisers of \eqref{maximisers} was arbitrary, \eqref{e:5,1} and \eqref{e:5,2} both follow.  
\end{proof}

\begin{lemma}\label{randomfunctional} Let $K\subset\RR^2$ be a deterministic convex compact set. Then the maximizer
\begin{equation*}
(x_0,y_0)=\arg\max_{(x,y)\in K}ax+by
\end{equation*}
is unique for all but countable many points $(a,b)$ on the unit sphere in $\RR^2$. Furthermore, if 
$(a,b)$ us chosen uniformly at random from the unit sphere in $\RR^2$, then 
\begin{equation}\label{e:p1}
\prob\left[\kappa\left(\partial K,(x_0,y_0)\right)\leq k\right]\leq\dfrac{k\cdot l}{2\pi},
\end{equation}
where $\kappa\left(\partial K,(x_0,y_0)\right)$ is the curvature of the boundary of $K$ at the point $(x_0,y_0)$, and where  $l$ denotes the length of the boundary of $K$.
\end{lemma}

\begin{proof}
The first part of the lemma is well known. The mapping 
\begin{equation*}
H:\,(a,b)\mapsto(x_0,y_0):=\arg\max_{(x,y)\in K} ax+by
\end{equation*}
is thus well defined for all but a countable number of points $(a,b)$ on the unit circle. If the interior of $K$ is empty, then $K$ lives on a line segment. The maximiser $(x_0,y_0)$ is then one of the two endpoints of this segment for almost all $(a,b)$, and since the curvature is infinite at these points, the claim of the lemma is trivially true. 

If $K$ has nonempty interior, then its boundary $\partial K$ is locally the graph of a convex function, and hence it is continuous. Spherical projection with respect to an interior point defines a parametrization $u(\theta)$ of $\partial K$,  with $\theta$ running on the unit circle. Since the boundary $\partial K$ is locally the graph of a convex function, $u(\theta)$ is differentiable everywhere except at a countable number of points, and it is twice differentiable everywhere except on a set of Lebesgue measure 0, see e.g. \cite[Theorem 1, page 242]{EvansGar}. The length $l=\int_{0}^{2\pi}\|\diff u(\theta)/\diff\theta\|_2\diff \theta$ is thus well defined and finite, and so is the reparametrization $c(t)$ of $u(\theta)$ with respect to the length $t=\int_0^{\theta}\|\diff u(\tau)/\diff\tau\|_2\diff \tau$. Furthermore, we have $\|\dot{c}(t)\|_2=1$ for all $t\in[0,l]$, and $\ddot{c}(t)$ is defined except on a Lebesgue-null set. Let $A$ be the subset of $t\in[0,l]$ where $\dot{c}(t)$ is defined, and $B$ the subset where $\ddot{c}(t)$ is defined. Without loss of generality, we may assume that the orientation of the curve $c(t)$ is positive, so that $G(t)=i \dot{c}(t)$ is the unit normal vector to $K$ at $c(t)$ (orthogonal to $\dot{c}(t)$ and pointing away from $K$). This defines a mapping $t\mapsto G(t)$ from $A$ to the unit circle. We make the following two observations:
\begin{enumerate}
\item[(a)] $\kappa(t):=\kappa(\partial K, c(t))=\|\ddot{c}(t)\|_2=\|\dot{G}(t)\|_2$ equals the curvature of $\partial K$ at $c(t)$.
\item[(b)] Given $(a,b)$ on the unit circle, if $c(t)=\arg\max_{(x,y)\in K}ax+by$ for some $t\in A$, and if this 
is the unique maximizer, then $(a,b)=G(t)$. 
\end{enumerate}
Let $T:=\{t\in [0,l]:\,\kappa(t)\leq k\}$. The fact that $G(t)$ is defined at all points where $\kappa(t)$ is  defined combined with Observations (a) and (b) imply that 
\begin{align}
\prob\left[\kappa(t)\leq k\right]&=\prob[G(T)]\nonumber\\
&=\int_{T}|\dot{G}(t)|\frac{dt}{2\pi}\label{tag *}\\
&=\int_{T}k(t)\frac{dt}{2\pi},\nonumber\\
&\leq\frac{k\cdot l}{2\pi}\label{voila}.
\end{align}
Equation \eqref{voila} establishes the claim \eqref{e:p1} of the lemma. The only nontrivial step that needs further explanation is \eqref{tag *}. Let $g:[0,l]\rightarrow[0,2\pi]$ be such that $G(t)=\exp(i\,g(t))$. Then $g(t)$ is well defined except at a countable number of points. By convexity of $K$, $g(t)$ is a non-decreasing function, and without loss of generality we may assume that it is right continuous. Equation \eqref{tag *} can thus be reformulated as follows, 
\begin{equation}\label{tag **}
\mu[g(T)]=\int_{T}\frac{\dot{g}(t)}{2\pi}\diff t, 
\end{equation}
where $\mu$ is the uniform probability measure on the interval $[0,2\pi]$. If $g$ is smooth and increasing, then \eqref{tag **} is simply a change of variable formula. In the general case we can approximate using smooth functions.  Thus take a standard mollifier $\phi_{\epsilon}$ and $g_{\epsilon,\delta}=(g+\delta h)\star \phi_{\epsilon}$ where $h(x)=x$.  The rationale for taking $g+\delta h$ is to render the derivative positive and $g_{\epsilon}$ increasing.  Equation \eqref{tag **} is true for $g_{\epsilon,\delta}$, and its general validity is obtained by first passing $\epsilon$ to zero, followed by $\delta$.   
\end{proof}
 
\newpage

\section{Proofs of The Main Theorems}\label{s:thms}

\subsection{Proof of Theorem \ref{t:1}}

\begin{proof}
Let $R:\mathcal{A}^*\times\mathcal{A}^*\rightarrow\RR$ be a symmetric scoring function, and consider the event 
\begin{equation*}
A^n(R)=\left\{\left|\lambda_R-\frac{L_{n,R}}{n}\right|\leq \frac{5|R|\sqrt{\ln (\e n)}}{\sqrt{n}} \right\}.
 \end{equation*}
It follows from Lemma \ref{l:3} and Theorem~\ref{t:AH} that 
\begin{align*}
\prob\left[A^n(R)\right]&\geq 
\prob\left[\left|\dfrac{L_{n,R}}{n}-\frac{E[L_{n,R}]}{n}\right|\leq 2|R|\sqrt{\ln(\e n)/n},\;
\left|\frac{E[L_{n,R}]}{n}-\lambda_R\right|\leq 3|R|\sqrt{\ln(\e n)/n}\right]\\
&\geq 1-n^{-\ln n},\quad\forall\,n\gg 1. 
\end{align*}
By the assumptions of the theorem, $x_0=\lambda_S$ and 
\begin{equation}\label{curvature condition}
\kappa(\partial\text{SET}_{S,T},(x_0,y_0))\geq k>0. 
\end{equation}
For small $\epsilon>0$ the point $P_{\epsilon}:=(x_{\epsilon},y_{\epsilon})$ on the boundary $\partial\text{SET}_{S,T}$ with $y$-coordinate $y_{\epsilon}:=y_0+\epsilon/k$ nearest to $(x_0,y_0)$ is well defined. Choose $a_{\epsilon}$ such that the linear form $f_{(1,a_{\epsilon})}:(x,y)\mapsto x+a_{\epsilon}y$ has its maximizer over the set $\text{SET}_{S,T}$ at $P_{\epsilon}$. This implies that for any $(x,y)\in \text{SET}_{S,T}$, 
\begin{equation}\label{e:9}
x+a_{\epsilon}y \leq x_{\epsilon}+a_{\epsilon}y_{\epsilon}.
\end{equation}
The curvature condition \eqref{curvature condition} implies that for all $\epsilon$ small enough, 
\begin{equation*}
x_{\epsilon}\leq x_0-\frac{\epsilon^{2}}{3}.
\end{equation*}
Combined with \eqref{e:9} for $(x,y)=(x_{0},y_{0})$, this yields  $(x-x_{\epsilon})+a_{\epsilon}(y-y_{\epsilon})\leq 0$, and since furthermore $x_{0}>x_{\epsilon}$, it follows that 
\begin{equation}\label{e:8}
a_{\epsilon}\geq\dfrac{x_{0}-x_{\epsilon}}{y_{\epsilon}-y_{0}}\geq\dfrac{\epsilon k}{3}.
\end{equation}

If $A^n(S)$ holds, then for any optimal alignment $\pi$ relative to $S$ we have 
\begin{equation}
\label{I}
\left|\dfrac{S_\pi(X_{[1,n]},Y_{[1,n]})}{n} - x_0\right| \leq \dfrac{5|S|\sqrt{\ln(\e n)}}{\sqrt{n}},
\end{equation}
and similarly, if the event $A^n(S+a_{\epsilon} T)$ holds, then
\begin{equation}\label{pretag}
\left|\dfrac{L_{n,S+a_{\epsilon}T}}{n}- \lambda_{S+a_{\epsilon}T}\right|\leq \dfrac{5|S|\sqrt{\ln(\e n)}}{\sqrt{n}}.
\end{equation}
On the other hand, 
\begin{equation*}
\lambda_{S+a_{\epsilon}T}=\max_{(x,y)\in\text{SET}_{S,T}}f_{(1,a_{\epsilon})}(x,y)=x_{\epsilon}+a_{\epsilon}y_{\epsilon},
\end{equation*}
and substituted into \eqref{pretag} this yields
\begin{equation}\label{tag{*}}
\left|\dfrac{L_{n, S+a_{\epsilon}T}}{n}-(x_{\epsilon}+a_{\epsilon}y_{\epsilon})\right|\leq \dfrac{5|S+a_{\epsilon}T|\sqrt{\ln(\e n)}}{\sqrt{n}}.  
\end{equation}

Next, for any optimal alignment $\pi$ relative to $S$, we have  
\begin{equation*}
\dfrac{(S+a_{\epsilon}T)_{\pi}
(X_{[1,n]},Y_{[1,n]})}{n}\leq\dfrac{L_{n, S+a_{\epsilon}T}}{n}
\stackrel{\text{\eqref{tag{*}}}}{\leq}
x_{\epsilon}+a_{\epsilon}y_{\epsilon}+\dfrac{5|S+a_{\epsilon}T| \sqrt{\ln(\e n)}}{\sqrt{n}}. 
\end{equation*}
It now follows from \eqref{I} that 
\begin{equation*}
a_{\epsilon}\left(\dfrac{T_{\pi}(X_{[1,n]},Y_{[1,n]})}{n}-y_{0}\right)
 \leq x_{\epsilon}-x_{0}+a_{\epsilon}(y_{\epsilon}-y_{0})+\frac{5(|S+a_{\epsilon}T|+|S|) \sqrt{\ln(\e n)}}{\sqrt{n}},
\end{equation*}
and since $x_{\epsilon}-x_{0}\leq 0<a_{\epsilon}$, this finally yields that for large $n$ and small  $\epsilon>0$, 
\begin{equation*}
\dfrac{T_{\pi}(X_{[1,n]},Y_{[1,n]})}{n}-y_{0}\leq\dfrac{\epsilon}{k}+\frac{5(|S+a_{\epsilon}T|+|S|) \sqrt{\ln(\e n)}}{a_{\epsilon}\sqrt{n}}\leq\dfrac{5(2|S|+a_{\epsilon}|T|)\sqrt{\ln(\e n)}}{a_{\epsilon}\sqrt{n}}.
\end{equation*}
In combination with with \eqref{e:8} this yields  
\begin{equation*}
\dfrac{T_{\pi}(X_{[1,n]},Y_{[1,n]})}{n}-y_{0}\leq\dfrac{\epsilon}{k}+\dfrac{5(6|S|+\epsilon k|T|) }
{\epsilon k}\sqrt{\frac{\ln(\e n)}{n}}.
\end{equation*}
For large $n$, we can minimize the right-hand side over $\epsilon$, yielding 
\begin{equation*}
\frac{T_{\pi}(X_{[1,n]},Y_{[1,n]})}{n}-y_{0}\le \frac{5|T|+2\sqrt{30|S|}}{k}\left(\frac{\ln(\e n)}{n} \right)^{1/4}
\end{equation*}
By changing the scoring function $T$ to $-T$, an analogous argument also shows that 
\begin{equation*}
-\left(\dfrac{T_{\pi}(X_{[1,n]},Y_{[1,n]})}{n}-y_{0}\right)\leq
\dfrac{5|T|+2\sqrt{30|S|}}{k}\left(\dfrac{\ln(\e n)}{n}\right)^{1/4},
\end{equation*}
and hence, 
\begin{equation}\label{the above inequality}
\left|\dfrac{T_{\pi}(X_{[1,n]},Y_{[1,n]})}{n}-y_{0}\right|\leq
\dfrac{5|T|+2\sqrt{30|S|}}{k}\left(\dfrac{\ln(\e n)}{n}\right)^{1/4}.
\end{equation}
We conclude that if all of the events $A^{n}(S)$ and $A^{n}(S+a_{\epsilon}T)$ and $A^{n}(S-a_{\epsilon}T)$ hold, then \eqref{the above inequality} applies, and since the probability that any individual event fails to hold is bounded by $n^{-\ln n}$, the claim of the theorem follows. 
\end{proof}

\subsection{Proof of Theorem \ref{t:2}}

\begin{proof}
Let $V=\arccos\langle S,T\rangle_F$, where $\langle\cdot,\cdot\rangle_F$ is the inner product on the space of symmetric scoring functions that corresponds to the Frobenius norm. Then $V$ is uniformly distributed on $[-\pi/2,\pi/2]$. Let $T_1$ be the Gram-Schmidt orthogonalization of $T$ with respect to $S_1:=S$, and let $U$ be a uniform random variable on $[0,2\pi]$, independent of $S$ and $T$, and hence also of $V$, and let us define $(S_2,T_2)=\Phi(S_1,T_1)$, where $\Phi$ is the rotation  
\begin{align*}
\Phi:\,\RR^2&\rightarrow\RR^2,\\
(x,y)&\mapsto(\cos(U)x+\sin(U)y,\, -\sin(U)x+\cos(U)y)
\end{align*}
by the angle $U$. It is easy to see that $\text{SET}_{S_2,T_2}=\Phi(\text{SET}_{S_1,T_1})$, and that 
under $\Phi^{-1}$, the point where $\text{SET}_{S_2,T_2}$ has a point of maximal first coordinate corresponds to the point where the random linear form $f:(x,y)\mapsto\cos(U)x+\sin(U)y$ takes 
a maximum value on $\text{SET}_{S_1,T_1}$. Furthermore, since $\Phi$ is angle-preserving, the 
curvature $\kappa_1$ of $\partial \text{SET}_{S_2,T_2}$ and $\partial \text{SET}_{S_1,T_1}$ at these points is  also the same. Lemma \ref{randomfunctional} applies, and we have 
$\prob[\kappa_1\leq k]\leq k\cdot l/(2\pi)$, where $l$ is the length of the boundary of $\text{SET}_{S_1,T_1}$. Since the scoring functions under considerations have unit norm, the rescaled alignment score cannot exceed $2$, implying that $l\leq 8$ and 
\begin{equation}\label{kappa2kappa}
\prob[\kappa_1\leq k]\leq\dfrac{4k}{\pi}.
\end{equation}

It remains to relate $\kappa_1$ to the curvature $\kappa$ of $\partial\text{SET}_{S,T}$ at the point where its first coordinate is maximized. Since $\text{SET}_{S,T}=\Psi(\text{SET}_{S_1,T_1})$, where $\Psi$ is the linear transformation 
\begin{align*}
\Psi:\,\RR^2&\rightarrow\RR^2,\\
(x,y)&\mapsto(x,\, \cos(v)x+\sin(V)y), 
\end{align*}
we have $\kappa=\kappa_1/|\sin V|\geq\kappa_1$, so that 
\begin{equation*}
\prob[\kappa<k]\leq\prob[\kappa_2<k]\leq\dfrac{4k}{\pi}, 
\end{equation*}
as claimed in the statement of the theorem. 
\end{proof}


\begin{thebibliography}{1}

\bibitem{azuma}
K.~Azuma.
\newblock Weighted sums of certain dependent random variables.
\newblock {\em Tohuku Math.\ J.}, 19:357--367, 1967.

\bibitem{Sankoff1}
V{\'a}cl{\'a}v Chvatal and David Sankoff.
\newblock Longest common subsequences of two random sequences.
\newblock {\em J.\ Appl.\ Probability}, 12:306--315, 1975.

\bibitem{EvansGar}
Lawrence~C. Evans and Ronald~F. Gariepy.
\newblock {\em Measure theory and fine properties of functions}.
\newblock Studies in Advanced Mathematics. CRC Press, Boca Raton, FL, 1992.

\bibitem{geometry}
Raphael Hauser and Heinrich Matzinger.
\newblock Distribution of aligned letter pairs in optimal alignments of random
  sequences.
\newblock {\em arXiv:1211.5491}, 2013.

\bibitem{needleman-wunsch}
S.B.\ Needleman and C.D. Wunsch.
\newblock A general method applicable to the search for similarities in the
  amino acid sequence of two proteins.
\newblock {\em Journal of Molecular Biology}, 48:443–453, 1970.

\bibitem{Vingron}
M.S.\ Waterman and M.~Vingron.
\newblock Sequence comparison significance and poisson approximation.
\newblock {\em Statistical Science}, 9(3):367--381, 1994.

\end{thebibliography}
\end{document}